\documentclass[12pt]{amsart}
\usepackage{amsmath}
\usepackage{amsfonts}
\usepackage{amssymb}
\usepackage{epsfig}
\usepackage{epstopdf}
\usepackage{color}
\usepackage[pdftex]{hyperref}

\usepackage[margin=1.2in]{geometry}

\newtheorem{theorem}{Theorem}
  \newtheorem{lemma}[theorem]{Lemma}

   \newtheorem{corollary}[theorem]{Corollary}

  \newtheorem{lemma-def}[theorem]{Lemma/Definition}
  
  \theoremstyle{definition}
  
   \newtheorem{remark}[theorem]{Remark}

\newcommand{\kk}{\mathbf k}

\newcommand{\CC}{\mathbb C}
\newcommand{\RR}{\mathbb R}
\newcommand{\QQ}{\mathbb Q}
\newcommand{\ZZ}{\mathbb Z}

\DeclareMathOperator{\spann}{span_\QQ}

\DeclareMathOperator{\conv}{conv}

\date{\today}
 
 \title[Do most polynomials generate a prime ideal?]{Do most polynomials generate a prime ideal?}
  \author{Josephine Yu}
 \address{School of Mathematics, Georgia Institute of Technology, Atlanta GA 30332, USA}
 \email {jyu@math.gatech.edu}
 \thanks {\emph {2010 Mathematics Subject Classification:} 13F20,14M25,14T05} 
\thanks{\emph{Keywords:} sparse polynomials, prime ideals, Newton polytopes, tropical geometry}

\date{\today}

\begin{document}

\begin{abstract}
For which monomial supports do most polynomials generate a prime ideal?  We give necessary and sufficient conditions for the radical of the ideal to be prime over an algebraically closed field.  In characteristic zero, the same conditions give primeness.  As an application we show that under the same combinatorial conditions on Newton polytopes, the stable intersection of tropical hypersurfaces is connected through codimension one.
 \end{abstract}

 \maketitle

\section{Introduction}



Let us fix monomial supports for some polynomials and ask: Do the polynomials generate a proper prime ideal in a polynomial ring for {\em most} choices of coefficients, over an algebraically closed field?
 
The answer is not always yes.  Here are some simple examples to illustrate what can go wrong.
\begin{enumerate}
\item[(1)] The polynomials may contain a non-trivial monomial factor; e.g.\ every polynomial of the form $a x+ bx^2 + cxy \in \CC[x,y]$ is reducible. Thus it is natural for us to work in the ring of Laurent polynomials instead of the usual polynomial ring.
\item[(2)] There may be more polynomials than variables; e.g.\ for most choices of coefficients the polynomials $ax+by+c, dx+ey+f, gx+hy+i \in \CC[x^{\pm 1},y^{\pm 1}]$ generate the unit ideal.
\item[(2')] We may be in case (2) after a monomial change of coordinates; e.g.\
 for most choices of coefficients the polynomials $axz+byw+c, dxz+eyw+f, gxz+hyw+i \in \CC[x^{\pm 1},y^{\pm 1},z^{\pm 1},w^{\pm 1}]$ generate the unit ideal, although the number of polynomials is less than the number of variables.
\item[(3)] The polynomials may form a zero-dimensional system of degree $> 1$; e.g.\ for most choices of coefficients the polynomials $ax + by + c, dx + ey + fxy \in \CC[x^{\pm 1},y^{\pm 1}]$ generate a zero-dimensional ideal of degree $2$, which is not prime.
\item[(3')] We may be in case (3) after a monomial change of coordinates; e.g. for most choices of coefficients, the ideal generated by $axz + by + c, dxz + ey + fxyz \in \CC[x^{\pm 1},y^{\pm 1}, z^{\pm 1}]$ is not prime, although it is not zero-dimensional.
\item[(4)] A subset of the polynomials may be in case (2') or (3'); e.g.\ for most choices of coefficients, the ideal generated by $axz + by + c, dxz + ey + fxyz, x+y+z+w \in \CC[x^{\pm 1},y^{\pm 1}, z^{\pm 1}, w^{\pm 1}]$ is not prime.
\end{enumerate}

Our main result, Theorem~\ref{thm:main}, gives a precise combinatorial characterization of these obstructions and shows that these are in fact the only ones for irreducibility, and even primeness in characteristic zero.

\medskip

Sufficient conditions for primeness in characteristic $p > 0$ are less tidy and we do not attempt to list them all here.  Here are some obstructions:
\begin{enumerate}
\item[(5)] Monomials in one of the polynomials may all be $p$-th powers; e.g.\ the ideal generated by $a x^p + b y^pz^{2p} + c$ is not prime over characteristic $p > 0$ for general coefficients $a,b,c$.
\item[(5')] We may be in the previous case after some cancellation; e.g.\ the ideal generated by $a x^p + b y^p + cz$ and $d y^p + e w^p + f z$ is not prime over characteristic $p>0$ for general coefficients $a,b,c,d,e,f$. 
\end{enumerate}

We state and prove the main result in the next section and discuss an application to tropical geometry in Section~\ref{sec:tropical}.
\medskip

\section{Main Theorem}
\label{sec:main}

Let $\kk$ be an algebraically closed field and $R = \kk[x_1^{\pm 1},\dots, x_n^{\pm 1}]$ be the ring of Laurent polynomials in $n$ variables over $\kk$.  
Let $k \leq n$, and fix $A_1, \dots, A_k \subset \ZZ^n$. Consider the space $(\kk^*)^{A_1} \times \cdots \times (\kk^*)^{A_k}$ of polynomials $(f_1,\dots,f_k)$ whose supports are $A_1,\dots, A_k$ respectively.  We say that a property holds for \emph{general} elements in $(\kk^*)^{A_1} \times \cdots \times (\kk^*)^{A_n}$ if it holds on a Zariski open dense subset.  


We work in the Laurent polynomial ring to disregard monomial factors from the polynomials.  
By translating each $A_j$ by any of its members, and equivalently factoring out a Laurent monomial from $f_j$, if necessary, we will assume for the rest of the paper that $0 \in A_j$ for all $j \in [k] := \{1,\dots,k\}$.

Let us first ask: When do general polynomials generate a proper ideal?
\begin{lemma}
\label{lem:proper}
Suppose $1 \leq k \leq n$, $A_1,\dots,A_k \subset \ZZ^n$ and $0 \in A_j$ for all $j$. General polynomials $f_1,\dots,f_k$ with supports $A_1,\dots,A_k$ respectively generate a proper ideal in the Laurent polynomial ring $R$ if and only if one, hence both, of the following equivalent conditions hold.
\begin{enumerate}
\item For every non-empty subset $J \subseteq [k]$, $\dim \spann \bigcup_{j \in J} A_j \geq |J|$.
\item There exist $u_1 \in A_1, \dots, u_k \in A_k$ that are linearly independent.
\end{enumerate}
\end{lemma}

\begin{proof}
Since $\kk$ is algebraically closed, by Hilbert's Nullstellensatz, the ideal $\langle f_1,\dots,f_k \rangle$ in $R$ is proper if and only if  the system of equations $f_1 = \cdots = f_k = 0$ has a solution in $(\kk^*)^n$.  This happens for general $f_1,\dots,f_k$ if and only if the sparse toric resultant of $(A_1,\dots,A_k)$ is identically zero, or equivalently the resultant variety has codimension zero.  The conditions (1) and (2) follow from the codimension formulas from~\cite[Theorem 1.1]{Sturmfels94} and~\cite[Theorem 2.23]{tropRes} respectively. The equivalence of (1) and (2) also follows directly from Perfect's generalization of Hall's marriage theorem~\cite{Perfect}, which has an easy combinatorial proof, so in fact we only need one of~\cite{Sturmfels94, tropRes}.
\end{proof}

We say that the collection $(A_1,\dots,A_k)$ contains an {\em independent transversal} if it satisfies (1) or (2) above. 
When $k=n$, this is precisely the condition for the mixed volume of $\conv(A_1),\dots,\conv(A_k)$ to be non-zero.  

\medskip

Let us now turn to the question of primeness.
We first introduce a combinatorial condition stronger than having an independent transversal.

\begin{lemma-def}
\label{lem:dragon}
The following  are  equivalent for any $A_1,\dots,A_k \subset \ZZ^n$. 
\begin{enumerate}
\item For every non-empty subset $J \subseteq [k]$, $\dim \spann \bigcup_{j\in J}
  A_j \geq |J|+1$.
\item For every $1 \leq j \leq k$, there exist $v_1 \in
  A_1,\dots,v_{j-1}\in A_{j-1}$, and $u_1,u_2 \in A_j$ such that $v_1,\dots,v_{j-1},u_1,u_2$ are linearly independent.
\end{enumerate}
We will call this the \emph{dragon marriage independent
transversal}  condition, or  DMIT for short.
\end{lemma-def}
When the disjoint union of $A_1,\dots,A_k$ is linearly independent,
we obtain the dragon marriage condition of Postnikov~\cite{Postnikov}.

\begin{proof}
For every non-empty $J \subseteq [k]$, by taking $j$ to be the largest
element in $J$, we get $(2)\Rightarrow(1)$.
For the converse, let $u$ be any non-zero element in $A_j$ and let $\phi$ be a linear projection of $\QQ^n$ onto an
$n-1$ dimensional subspace such that $\phi(u) = 0$.  By the
condition $(1)$, we have $\dim \spann \sum_{j\in J}
  \phi(A_j)\geq |J|$ for all non-empty $J \subseteq [k]$. By equivalence of (1) and (2) in Lemma~\ref{lem:proper}, there exist
  $v_1 \in A_1,\dots, v_j \in A_j$ such that
  $\phi(v_1),\dots,\phi(v_j)$ are linearly independent.  It follows
  that $v_1,\dots,v_j,u$ are linearly independent.
\end{proof}

We can now prove the main theorem:

\begin{theorem}
\label{thm:main}
Suppose $1 \leq k \leq n$, $A_1,\dots,A_k \subset \ZZ^n$ and $0 \in A_j$ for all $j$.  General polynomials $f_1,\dots,f_k$ with supports $A_1,\dots,A_k$ generate a proper ideal whose radical is prime in $R$ if and only if for every non-empty subset $J \subseteq [k]$ one of the following holds:
\begin{enumerate}
\item $\dim \spann \bigcup_{j\in J} A_j \geq |J|+1$, or
\item $\dim \spann \bigcup_{j\in J} A_j = |J|$ and the mixed volume of $(\conv(A_j))_{j\in J}$ is~$1$.
\end{enumerate}
If $\kk$ has characteristic $0$, then ``whose radical is prime'' can be replaced with ``is prime''.
\end{theorem}


\begin{remark}
In characteristic $p$ it is not enough to assume that the monomials are not all $p$th powers modulo monomial factors.  See example (5') in the Introduction.
\end{remark}




\begin{proof}
We will first check that the condition is necessary.  If for some non-empty $J \subseteq [k]$ we have $\dim \spann \bigcup_{j \in J} A_j < |J|$, then the general polynomials generate the unit ideal by Lemma~\ref{lem:proper}.  

Now suppose $\dim \spann \bigcup_{j \in J}A_j \geq |J|$ for all $J \subseteq [k]$, but there is an $K \subseteq [k]$ for which $\dim \spann \bigcup_{j \in K} A_j = |K|$ and the mixed volume of $(\conv(A_j))_{j\in K}$ is greater than $1$.  After  a monomial change of coordinates, we may assume that $\spann \bigcup_{j \in K}A_j$ is a coordinate subspace of $\QQ^n$ and that the polynomials $(f_j)_{j \in K}$ involve only the variables $x_1,\dots,x_{|K|}$.  Then by the Bernstein–-Khovanskii–-Kushnirenko theorem~\cite{Bernstein}, general polynomials $(f_j)_{j \in K}$ have more than one, but finitely many, distinct common roots in $(\kk^*)^{|K|}$. Then
\begin{align*}\langle f_1,\dots,f_k \rangle &= \langle f_j : j \in K \rangle + \langle f_j : j \notin K \rangle \\ &= (I_1 \cap \cdots \cap I_m) + \langle f_j : j \notin K \rangle \\ &= (I_1 + \langle f_j : j \notin K\rangle) \cap \cdots \cap (I_m + \langle f_j : j \notin K\rangle)
\end{align*}
where $m>1$ is the mixed volume of $(\conv(A_j))_{j \in K}$ and $I_1,\dots,I_m$ are the ideals defining the $m$ distinct common roots of $f_1,\dots,f_k$. For each of these roots in $(\kk^*)^{|K|}$, substituting it into the $x_1,\dots,x_{|K|}$ coordinates of $(f_j)_{j \notin K}$ gives general polynomials in the remaining coordinates that generate a proper ideal, because they contain an independent transversal as in Lemma~\ref{lem:proper}.  This shows that the ideals $I_i + \langle f_j : j \notin K\rangle$ for $1\leq i \leq m$ are proper and their zero sets do not contain each other; thus the radical of $\langle f_1,\dots,f_k \rangle$ is not prime.

\medskip

For the converse, suppose $K$ is the largest (possibly empty) subset of $[k]$ satisfying (2).  As above, we may assume that $(f_j)_{j\in K}$ involve only the variables $x_1,\dots,x_{|K|}$.  Since the mixed volume is $1$, the polynomials $(f_j)_{j \in K}$ have a unique common root in $(\kk^*)^{|K|}$.  Then the ideal $\langle f_1,\dots,f_k \rangle$ is isomorphic to the ideal generated by $(f_j)_{j \notin K}$ after substituting this common root into variables $x_1,\dots,x_{|K|}$.  This substitution gives general polynomials whose supports are linear projections of $(A_j)_{j \notin K}$ onto the remaining coordinates.  Every non-empty collection of the projected supports satisfy condition (1).  Otherwise the union of such a collection and $K$ would be a larger set satisfying (2).  Thus we have reduced the problem to the case when $(A_1,\dots,A_k)$ satisfies DMIT.

We now recall a version of Bertini theorem from~\cite[Theorem 6.3 (4)]{Jouanolou}. For $u = (u_0,u_1,\dots,u_n) \in \kk^{n+1}$, let $H_u$ be the hyperplane in $\kk^n$ defined by $u_0 + u_1 y_1 + \cdots c_n y_n = 0$ where $y_i$'s are coordinates in $\kk^n$.  The theorem says, for the case when $\kk$ is algebraically closed:


\begin{quote}Let $X$ be a $\kk$-scheme of finite type; $\alpha : X \rightarrow \kk^n$ be a $\kk$-morphism.  
\begin{enumerate}
\item If $\dim \overline{\alpha(X)} \geq 2$ and $X$ is irreducible, then  $\alpha^{-1}(H_u)$ is irreducible for general \mbox{$u\in\kk^{n+1}$}.
\item If $\kk$ has characteristic zero and $X$ is smooth, then $\alpha^{-1}(H_u)$ is smooth for general  \mbox{$u\in\kk^{n+1}$}.\footnote{For example, in characteristic $p>2$ the curves in $\kk^2$ defined by $a x^p + b y^2 + c y + d = 0$ are not smooth for general coefficients $a,b,c,d$~\cite{Zariski}. The morphism $\alpha$ in this case is $\alpha : \kk^2 \rightarrow \kk^3$ given by $(x,y) \mapsto (x^p, y^2,y)$.}
\end{enumerate} 
\end{quote}

Let $I = \langle f_1, \dots, f_{k-1}\rangle$. If $k=1$, let $I$ be the zero ideal.   Let $X$ be the variety of $I$ in $(\kk^*)^n$. By induction, we may assume that $X$ is irreducible. 
Consider the monomial map $\alpha : X \rightarrow \kk^{|A_k|}$ given by $\alpha(x) = (x^{a_1},\dots,x^{a_m})$ where $A_k = \{a_1,\dots,a_m\}$.
We need to show that the image of $X$ under $\alpha$ has dimension at least two. 

Since $(A_1,\dots,A_k)$ satisfies the DMIT condition, there exist $v_1\in A_1, \dots$, $v_{k-1}\in A_{k-1}$ and $u_1,u_2 \in A_k$ that are linearly independent.  Therefore the collection $(A_1,\dots,A_{k-1}$, $\{u_1,0\}$, $\{u_2,0\})$ contains an independent transversal. By Lemma~\ref{lem:proper}, for general $(c_1,c_2) \in \kk^2$, the system $f_1 = \cdots =f_{k-1} = x^{u_1}-c_1 = x^{u_2}-c_2 = 0$ has a solution in $(\kk^*)^n$.  It follows that the image of the map $X \rightarrow \kk^2$ given by $x \mapsto (x^{u_1},x^{u_2})$ is dense in $\kk^2$.   Thus $\dim \overline{\alpha(X)} \geq 2$. 

Then by the Bertini theorem (1) above, $\alpha^{-1}(H_u) \subset X$ is irreducible for general hyperplanes $H_u$.  But $\alpha^{-1}(H_u)$ is precisely the intersection of $X$ with the hypersurface in $(\kk^*)^n$ defined by $f_k = 0$.
Note that we need the assumption that $A_k$ contains $0$ because general hyperplanes $H_u$ are defined by equations with a constant term $u_0$.

In characteristic zero, by induction and Bertini theorem (2) above,  $\alpha^{-1}(H_u) \subset X$ is also smooth.  It follows that the ideal $\langle f_1,\dots,f_k \rangle$ is prime.
\end{proof}

\section{Stable intersection of tropical hypersurfaces}
\label{sec:tropical}
I want to point out, to readers familiar with tropical geometry, a consequence of the main theorem on connectivity of stable intersection of tropical hypersurfaces.  Please refer to~\cite{MaclaganSturmfels, stableIntersection} for background on tropical hypersurfaces and their stable intersections.

\begin{corollary}
\label{cor:connected}
Suppose the collection $(A_1,\dots,A_k)$ satisfies the condition in Theorem~\ref{thm:main}.  Then for any tropical polynomials with corresponding supports, the stable intersection of their tropical hypersurfaces is connected through codimension one.
\end{corollary}

\begin{proof}
The stable intersection of tropical hypersurfaces is the tropicalization of an ideal generated by some general polynomials with the given Newton polytope~\cite{OssermanPayne, stableIntersection}, over any algebraically closed field. By Theorem~\ref{thm:main} above, the zero set of this ideal is irreducible.  The result then follows from~\cite{CartwrightPayne}.
\end{proof}

This result has applications in computational tropical geometry, where we need connectivity for graph traversal algorithms, such as those in~\cite{gfan, Malajovich}, to work correctly.
In the constant-coefficient case, when the tropical hypersurfaces are fans, the stable intersection is connected through codimension one without any assumption on the supports~\cite{tropRes, stableIntersection}. 

\section{Open Problems}

What are precise combinatorial conditions for primeness in positive characteristic?  See Example (5') in the Introduction.

Although Corollary~\ref{cor:connected} is a purely combinatorial statement about mixed subdivisions of polytopes and their dual complexes, I do not know of a combinatorial proof, even for the case of two tropical hypersurfaces in $\RR^3$.  It will be interesting to find a direct proof and also to study combinatorial and topological properties of the stable intersection, such as shellability.

 The DMIT condition is checkable in polynomial time in the size of the input $(A_1,\dots,A_k)$, as follows. By the proof of Lemma~\ref{lem:dragon}, the DMIT condition holds if and only if for every $j \in [k]$ and
 for every $u \in A_j$, $(\phi_u(A_1), \dots, \phi_u(A_j))$ contains an independent transversal, where
 $\phi_u$ is a linear projection from $\QQ^n$ to an $n-1$ dimensional subspace such that $\phi_u(u) = 0$.  The existence of independent
 transversal can be checked in polynomial time using matroid
 intersection as in~\cite[Theorem 2.29]{tropRes}.  
I do not know whether the condition of Theorem~\ref{thm:main}, which is weaker than DMIT, is checkable in polynomial time.  In particular, can we efficiently check whether the mixed volume is equal to $1$?

We can study the (Zariski closure of) the locus of coefficients for which smoothness, reducedness, or irreducibility fails.  For the case when the number of polynomials is equal to the number of variables, the non-smooth locus is the mixed discriminant studied in~\cite{CCDDS}, which conincides with the $A$-discriminant of the Cayley sum of $A_1,\dots,A_k$.

\section*{Acknowledgments}
I thank Florian Enescu, Anton Leykin, Bernd Sturmfels, and Charles Wang for helpful discussions, Carlos D'Andrea for pointing me to the reference~\cite{Jouanolou}, and Dustin Cartwright and a referee 
for pointing out an error in the earlier version of the paper.  I acknowledge partial support from NSF-DMS Grant \#1101289.

\bibliographystyle{amsalpha}
\bibliography{mybib}
 \end{document}